\documentclass[12pt]{amsart}
\usepackage{amsmath,amssymb,amsthm,mathrsfs,color,dsfont}
\usepackage{hyperref}
\usepackage{enumerate}
\usepackage{xcomment}
\usepackage[mathlines]{lineno}
\modulolinenumbers[1]
\usepackage{comment}
\usepackage{marginnote}
\usepackage{graphicx}
\theoremstyle{definition}
\newtheorem{thm}{Theorem}[section]

\newtheorem{prop}{Proposition}[section]

\newtheorem{rem}{Remark}[section]
\numberwithin{equation}{section}

\theoremstyle{definition}
\newtheorem{definition}{Definition}[section]

\theoremstyle{remark}

\newcommand{\C}{{\mathbb C}}

\newcommand{\R}{{\mathbb R}}

\def\11{{\rm 1~\hspace{-1.4ex}l} }

\newcommand{\vertiii}[1]{{\left\vert\kern-0.25ex\left\vert\kern-0.25ex\left\vert #1
    \right\vert\kern-0.25ex\right\vert\kern-0.25ex\right\vert}}

\begin{document}
\title[Degeneracy and multiplicity in one-dimensional NLS]{Degeneracy and multiplicity of standing-waves of the one-dimensional non-linear Schr\"odinger equation for a class of algebraic non-linearities}

\author[D.~Garrisi]{Daniele Garrisi}
\address[Daniele Garrisi]{%
  Department of Mathematical Sciences\\
  University of Nottingham Ningbo China\\
  199 Taikang East Road\\
  315100, Ningbo, People's Republic of China}
\email{daniele.garrisi@nottingham.edu.cn}
            

\author[V.~Georgiev]{Vladimir Georgiev}
\address[Vladimir Georgiev]{%
Dipartimento di Matematica, Universit\`a di Pisa\\
Largo B. Pontecorvo 5, 56100 Pisa, Italy}
\address[Vladimir Georgiev]{%
Faculty of Science and Engineering \\ 
Waseda University \\
3-4-1, Okubo, Shinjuku-ku, Tokyo 169-8555 \\
Japan}
\address[Vladimir Georgiev]{%
Institute of Mathematics and Informatics, Bulgarian Academy of Sciences, Georgi Bonchev Str.\\
Block 8, 1113 Sofia\\
Bulgaria}
\email{georgiev@dm.unipi.it}%
\thanks{The work was supported by the University of Nottingham Ningbo China (New Researcher's Grant); 
``Problemi stazionari e di evoluzione nelle equazioni di campo non-lineari dispersive'' of GNAMPA 2020, the project PRIN 2020XB3EFL; the Italian Ministry of Universities and Research; the Top Global University Project, Waseda University; 
the Institute of Mathematics and Informatics, Bulgarian Academy of Sciences.}
\date{\today}
\begin{abstract}
We study the existence, the stability and the non-degeneracy of normalized
standing-waves solutions to a one dimensional non-linear Schr\"odinger equation.
The non-linearity belongs to a class of algebraic functions appropriately defined.
We can show that for some of these non-linearities 
one can observe the existence of degenerate minima, and the
multiplicity of positive, radially symmetric minima having the
same mass and the same energy. We also prove the stability of
the ground-state and the stability of normalized standing-waves whose profile is a minimum of the
energy constrained to the mass.
\end{abstract}
\maketitle
\section{Introduction}
In this work we classify the stability, uniqueness and non-degeneracy of
normalized standing-waves to the non-linear Schr\"odinger equation
in dimension one
\begin{linenomath}
\begin{equation}
\label{eq.NLS}
\tag{NLS}
i\partial_t\phi(t,x) + \Delta_x \phi(t,x) - G'(|\phi(t,x)|)\frac{\phi(t,x)}{|\phi(t,x)|} = 0,
\end{equation}
\end{linenomath}
where \(G\colon [0,+\infty)\to\mathbb{R}\) is a \(C^2\) non-linearity such that
\(G(0) = G'(0) = G''(0) = 0\). A standing-wave is a solution \(\phi\) to \eqref{eq.NLS} which can be written as 
\begin{linenomath}
\begin{equation}
\label{eq.sw}
\phi(t,x) = e^{-i\omega t} R(x),
\end{equation}
\end{linenomath}
where \(\omega\in\mathbb{R}\) and \(R\in H^1 (\mathbb{R};\mathbb{R})\). This function is sometimes called \textsl{profile}. 
A normalized standing-wave is a standing-wave 
such that \(E(R) = \inf_{S(\lambda)} E\) for some \(\lambda > 0\), where 
\begin{linenomath}
\begin{equation}
\label{eq.E}
E\colon H^1 (\mathbb{R};\mathbb{C})\to\mathbb{R},\quad
E(u) := \frac{1}{2}\int_{-\infty}^{+\infty} |u'(x)|\sp 2 dx +
\int_{-\infty}^{+\infty} G(|u(x)|)dx
\end{equation}
\end{linenomath}
and
\begin{linenomath}
\begin{equation*}
S(\lambda) := \bigg\{u\in H\sp 1(\mathbb{R};\mathbb{C})\bigg|\hskip 1pt M(u) := \lambda\bigg\},\quad\lambda > 0,
\end{equation*}
\end{linenomath}
where
\begin{linenomath}
\begin{equation}
\label{eq.M}
M(u) := \int_{-\infty}^{+\infty} |u(x)|^2 dx.
\end{equation}
\end{linenomath}
We study the following four problems: the orbital stability of the ground-state, the uniqueness, the orbital stability of the standing-waves and non-degeneracy of minima. We assume that \(G\) is such that 
\begin{linenomath}
\begin{equation}
\label{eq.G}
\tag{G}
G(s) = -\frac{s^2 V(s)}{2},
\end{equation}
\end{linenomath}
where \(V\) is the function satisfying the implicit equation
\begin{linenomath}
\begin{equation}
\label{eq.V}
\tag{V}
s^2 = aV^3 + bV^2 + cV,\quad V(0) = 0.
\end{equation}
\end{linenomath}
We introduce the notation
\begin{linenomath}
\begin{equation*}
\sigma := \frac{b^2}{ac}
\end{equation*}
\end{linenomath}
and require
\begin{linenomath}
\begin{equation}
\label{eq.C}
\tag{C}
a,c > 0,\quad\sigma < 3.
\end{equation}
\end{linenomath}
\subsubsection*{Orbital stability of the ground-state and existence of minima}
We call \textsl{ground-state} the set of minima of \(E\) constrained to \(S(\lambda)\) and use the notation
\begin{linenomath}
\begin{equation*}
\mathcal{G}_\lambda := \{u\in S(\lambda)\mid E(u) = \inf_{S(\lambda)} E\}.
\end{equation*}
\end{linenomath}
 In the space
\(X := H^1 (\mathbb{R};\mathbb{C})\) we define the metric \(d\)
induced by the scalar product
\begin{linenomath}
\begin{equation*}
(f,g)_{H^1} := \text{Re}\int_{-\infty}^{\infty} f'\overline{g'} dx + \text{Re}\int_{-\infty}^{\infty} f\overline{g}dx.
\end{equation*}
\end{linenomath}
For every \(u_0\in H^1\), for every \(t\geq 0\), we set
\(U(t,u_0) := \phi(t,\cdot)\), where \(\phi\) is the unique solution
to the initial value problem. In Proposition~\ref{prop.branch}, we prove that \(G\) is \(C^2\). Therefore,
\(G'\) is locally Lipschitz.
\begin{definition}
A subset \(S\subseteq X\) is \textsl{stable} if for every
\(\varepsilon > 0\) there exists \(\delta > 0\) such that, for every \(v\in X\), \(d(v,S) < \delta\)
implies \(d(U(t,v),S) < \varepsilon\) for every \(t\geq 0\).
\end{definition}
The definition given above is well-posed as long as the initial value problem for \eqref{eq.NLS} exists, is unique
and defined for every \(t\in [0,+\infty)\). This will be pointed out in the proof of Theorem~1.
\begin{thm}
\label{thm.stability.gs}
If \eqref{eq.G}, \eqref{eq.V} and \eqref{eq.C} hold, then 
for every \(\lambda > 0\), the set \(\mathcal{G}_\lambda\) is non-empty and stable.
\end{thm}
The stability of the ground-state in dimension one has been proved in \cite{GG17} under more general assumptions than the ones required in this work. In \S~\ref{sect.stability.gs}, we are going to check that a non-linearity satisfying 
\eqref{eq.G} and \eqref{eq.V}, also satisfies properties (G1), (G2a) and (G2b) of \cite{GG17}. Therefore, the stability
of \(\mathcal{G}_\lambda\) follows from \cite[Theorem~1.1]{GG17}.
\subsubsection*{Uniqueness of minima} 
By \(H^1 _{r,+} (\mathbb{R};\mathbb{R})\), we denote the space of
real-valued, positive and symmetric function. 
If \(u\in\mathcal{G}_\lambda\), then \(zu(\cdot + y)\) is also a minimum of \(E\) constrained to \(S(\lambda)\)
for every complex number \(z\) with \(|z| = 1\) and \(y\in\mathbb{R}\). 
Therefore, the set
\begin{linenomath}
\begin{equation}
\label{eq.gsu}
\mathcal{G}_\lambda (u) := 
\{zu(\cdot + y)\mid (z,y)\in S^1\times\mathbb{R}\}
\subseteq\mathcal{G}_\lambda
\end{equation}
\end{linenomath}
is always uncountable for every \(u\in\mathcal{G}_\lambda\). However,
one can prove that in every set \eqref{eq.gsu} there exists a unique
\(R\in H_{r,+} ^1\). For this fact, we refer to \cite{BJM09,Mar09} for the dimension \(N\geq 2\)
and to \cite[Lemma~2.4]{GG17} for the dimension \(N = 1\). In other words, the set
\(\mathcal{G}_\lambda (u)\cap H_{r,+}^1\) contains only one element. 
The problem of the uniqueness aims to calculate the cardinality of the set
\begin{linenomath}
\begin{equation*}
K_\lambda := \mathcal{G}_\lambda\cap H_{r,+}^1
\end{equation*}
\end{linenomath}
for a fixed \(\lambda > 0\). In the literature, there are not many non-linearities where one can evaluate or estimate the cardinality of 
\(K_\lambda\). When \(G(s) = -as^p\), the pure-power case, it is known
that \(|K_\lambda| = 1\) in every dimension. In fact, given \(u\in\mathcal{G}_\lambda\),
there exists \(\omega\) such that
\begin{linenomath}
\begin{equation*}
\Delta R  + aR ^{p - 1} - \omega R = 0,\quad R := |u|.
\end{equation*}
\end{linenomath}
By inspection, 
\begin{linenomath}
\begin{equation*}
u(x) = z(\omega a^{-1})^{\frac{1}{p - 2}}R_1(x),\quad 
a^{-\frac{2}{p - 2}}\omega^{\frac{6 - p}{2(p - 2)}}\|R_1\|_{L^2}^2 = \lambda,
\end{equation*}
\end{linenomath}
where \(R_1\) is the unique, positive, vanishing at infinity and symmetric
solution to the elliptic equation \(\Delta R_1  + aR_1 ^{p - 1} - R_1 = 0\). Although 
rescaling properties do not hold, uniqueness results for
the non-linear Schr\"odinger equation has been obtained
for the double power non-linearity \(G(s) = -as^p + bs^q\) with
\(2 < p < 6\) and \(p < q\) and dimension one in \cite{GG17}. In \cite{GG19} it has been proved that \(|K_\lambda| = 1\)
for every \(\lambda > 0\) for some triple power non-linearities.
\begin{thm}[Uniqueness]
\label{thm.uniqueness}
\

Let \(G\colon\mathbb{R}_+\to\mathbb{R}\) be a continuously differentiable function satisfying \eqref{eq.G}, \eqref{eq.V}
and \eqref{eq.C}. Therefore
\begin{enumerate}[(i)]
\itemsep=0.3em
\item if \(0 < \sigma \leq 2\), there exists a unique positive symmetric minimum, that is \(|K_\lambda| = 1\) for
every \(\lambda > 0\);
\item if \(2 < \sigma < 3\), there exists a unique level \(0 < \lambda_2 = \lambda_2 (\sigma)\) such that
\(|K_{\lambda_2}| = 2\). That is, there are two positive, symmetric minima of \(E\) constrained to
\(S(\lambda)\). If \(\lambda\neq\lambda_2\), again
\(|K_\lambda| = 1\).
\end{enumerate}
\end{thm}
\subsubsection*{Stability of standing-waves}
From \cite[Definition~2.3]{BBBM10}, a standing-wave \eqref{eq.sw} is stable if and only if the set \(\mathcal{G}_\lambda (R)\)
is stable. While the stability of the ground-state holds under mild assumptions described in \cite{GG17}, 
the stability of the set \eqref{eq.gsu} has been so far obtained
under very specific assumptions on the nonlinearity: pure powers in every dimension, \cite{CL82}, double-power non-linearities in dimension one, \cite{IK93,Oht95,Mae08,LTZ21,GG17}, coupled
Schr\"odinger systems where the non-linear term is a homogenous polynomial of degree four
\cite{NW11} in dimension one. In \cite{GG19,FH21}, an extension to a triple-power case has been done in some cases
Possible appearance of resonances for the linearized NLS was studied in \cite{CG16}.


\begin{thm}[Stability]
\label{thm.stability}
If \eqref{eq.G}, \eqref{eq.V} and \eqref{eq.C} holds, then for every \(\lambda > 0\), the set \(\mathcal{G}_\lambda\) is
non-empty and stable.
\end{thm}
\subsubsection*{Non-degeneracy}
We set \(S_r (\lambda) := S(\lambda)\cap H^1 _r \).
\begin{definition}
A critical point is non-degenerate if and only if the Hessian of \(E\) restricted to \(S_r (\lambda)\) is invertible. 
\end{definition}
If \(R\in S_r (\lambda)\) is a critical point of \(E\) constrained
to \(S_r (\lambda)\), then there exists \(\omega > 0\) such that
\begin{linenomath}
\begin{equation}
  \label{eq.elliptic}
  R''(x) - G'(R(x)) - \omega R(x) = 0.
  \end{equation}
  \end{linenomath}
Following the steps of \cite[\S~3]{GG17}, one can show that if \(R\) is a degenerate critical point, then there
exists $ d $ in $ \mathbb{R} $ and $ f $ in $ C^2\cap H^1 _r $
such that
\begin{linenomath}
\begin{equation*}
  L_+ (f) := -f'' + (G''(R(x)) + \omega)f = dR,\quad
  (f,R)_2 = 0.
\end{equation*}
\end{linenomath}
Examples of degenerate and non-degenerate standing-waves for
the non-linear Schr\"odinger equation 
already exist in the pure-power case $ G(s) = -as^6 $ given
in \cite{Wei85}, which provides an explicit construction of a solution
of such a function \(f\) in \cite[Proof~of~Proposition~2.10]{Wei85}.
For \(p < 6\), all the critical points are non-degenerate. 
However, in the mass-critical case \(p = 6\), the function \(E\) is not bounded below on \(S(\lambda)\) or the 
minimum is not achieved. This fact follows by applying a conformal rescaling, \cite[Proof~of~Lemma~3.1]{GG17}. Therefore, 
such \(R\) is not a minimum. The next result classifies the degeneracy of critical points and show that for some non-linearities and proves that degenerate \textsl{minima} do exist. 
\begin{thm}[Non-degeneracy]
\label{thm.nondeg}
Under the assumptions \eqref{eq.G}, \eqref{eq.V} and \eqref{eq.C},
\begin{enumerate}[(i)]
\itemsep=0.3em
\item if \(0 < \sigma\neq 2\), then minima of \(E\) over \(S_r (\lambda)\) are non-degenerate
for every \(\lambda > 0\);
\item if \(\sigma = 2\), there exists a unique level \(\lambda_d\) such that 
all the minima of \(E\) on \(S_r (\lambda_d)\) are degenerate.
\end{enumerate}
\end{thm}
In double or triple powers, \cite{GG17,GG19} and \cite{LTZ21}, phenomena, as degenerate minima or
multiplicity of minima with the same mass do not appear. In \cite{BS19}, it has been proved the existence of multiple critical points for the coupled non-linear Schr\"odinger equation. 
In \cite{Gar20}, uniqueness, stability and non-degeneracy have been studied for a non-linear Klein-Gordon equation; examples of degenerate minima and constraints where the energy functional has more than one minimum, have been shown. In the quoted reference, the non-linearity is a double power. None of the solutions to \eqref{eq.C} and \eqref{eq.G} and \eqref{eq.V} can
be written as a finite linear combination of pure powers. From \eqref{eq.G} and \eqref{eq.V}, it follows that 
the function \(G\) is also an algebraic. For technical reasons, Proposition~\ref{prop.lambda}, we put more emphasis
on \eqref{eq.V}.

The proof of the Theorem~\ref{thm.uniqueness}, Theorem~\ref{thm.nondeg} and Theorem~\ref{thm.stability} rely on the study of the function 
\begin{linenomath}
\begin{equation}
\label{eq.lambda.def}
\lambda(\omega) := \|R_\omega\|_{L^2}^2,
\end{equation}
\end{linenomath}
where \(R_\omega\) is the unique positive and symmetric function in \(H^1\) satisfying
\eqref{eq.elliptic}. In \S~\ref{sect.lambda}, we calculate the function \(\lambda\) for non-linearities satisfying
\eqref{eq.G} and \eqref{eq.V}. The next three sections address the proof of the theorems listed above: we prove Theorem~\ref{thm.stability.gs}, Theorem~\ref{thm.uniqueness} and Theorem~\ref{thm.stability} in \S~\ref{sect.stability.gs},
and Theorem~\ref{thm.nondeg} in \S~\ref{thm.nondeg}.
\section{Preliminary facts and notations}
\label{sect.lambda} The next results justify the introduction of the function \(V\) and the reason why our assumptions are expressed
in terms of \(V\) rather than the non-linearity \(G\).
\begin{prop}
\label{prop.ber-lio-str}
Suppose that \(V(0) = 0\). For every \(\omega\in (0,\sup(V))\),
there exists a unique solution \(R_\omega > 0\) to \eqref{eq.elliptic} in 
\(H^1_{r}(\mathbb{R})\) such that
\begin{enumerate}[(i)]
\itemsep=0.3em
\item \(R_\omega\) is positive and strictly symmetric-decreasing;
\item \(R_\omega (0) = \inf\{s > 0\mid V(s) = \omega\}\);
\item the map \(\phi\colon J\to H^1 _r\) defined as \(\phi(\omega) = R_\omega\) is \(C^1(J,H^1 _r(\mathbb{R};\mathbb{C}))\)
for every open interval \(J\subseteq (0,\sup(V))\).
\end{enumerate}
\end{prop}
The proof of (i) and (ii) follows from results and remarks of \cite[Theorem~5]{BL83a}. In their assumptions, \(R_\omega\) 
vanishes at infinity.
In our case, this fact follows from the inclusion \(H^1 (\mathbb{R})\subseteq L^{\infty}(\mathbb{R})\) and \eqref{eq.elliptic} which gives \(R_\omega\in W^{2,\infty}\). Therefore, \(\lim_{|x|\to\infty} R_\omega(x) = 0\); (iii) follows from 
\cite[Proposition~6]{Gar23}
and \cite{Sha83}.
\begin{prop}
\label{prop.branch}
If \(\sigma < 3\) and \(c > 0\), there is a unique \(C^2\) function \(V\) satisfying
\eqref{eq.V} and defined on \([0,+\infty)\) such that 
\begin{enumerate}[(i)]
\itemsep=0.3em
\item \(V'(s) > 0\) if \(s > 0\);
\item \(\lim_{s\to\infty} V(s) = +\infty\).
\end{enumerate}
\end{prop}
\begin{proof}
(i). We apply the Implicit Function Theorem to the \(F\colon\mathbb{R}\times\mathbb{R}\to\mathbb{R}\)
defined as \(F(x,y) = x - ay^3 - by^2 - cy\). Since \(\partial_x {F}(0,0) = 1\neq 0\), there exists \(\varepsilon > 0\)
and a smooth function \(\varphi\colon(-\varepsilon,\varepsilon)\to\mathbb{R}\) such that \(F(x,\varphi(x)) = 0\)
for every \(x\in (-\varepsilon,\varepsilon)\). We claim that \(\varphi\) can be extended to \((-\varepsilon,+\infty)\) 
as a continuously differentiable function. 
On the contrary, let 
\begin{linenomath}
\begin{equation}
\label{eq.branch.1}
\varphi\colon(-\varepsilon,\ell)\to\R
\end{equation}
\end{linenomath}
be the maximal interval of existence of a solution. 

Since 
\(x = \varphi(x)(a\varphi(x)^2 + b\varphi(x) + c)\) and \(\sigma < 4\), \(\varphi > 0\) on \((0,\ell)\).
Since \(\sigma < 3\), we have \(\partial_y F(x,y) < 0\). Since \(\varphi > 0\) on \((0,\ell)\), from 
\begin{linenomath}
\begin{equation}
\label{eq.ift}
\varphi'(x) = -\frac{\partial_y F(x,\varphi(x))}{\partial_x F(x,\varphi(x))}
\end{equation}
\end{linenomath}
we obtain \(\varphi' > 0\) on \((0,\ell)\). \(\varphi\) is also bounded as
\(F(x,\varphi(x)) = 0\), implies \(|a\varphi^3(x) + b\varphi^2(x) + c\varphi(x)|\leq\ell\) for every 
\(x\in (0,l)\). Since \(\varphi\) is strictly increasing, there exists the limit \(\lim_{x\to\ell}\varphi(x) = L\). Since
\(\partial_x F(u,L)\neq 0\), there exists \(\varepsilon_0 > 0\) and 
\(\psi\colon (\ell - \varepsilon_0,\ell +\varepsilon_0)\to\mathbb{R}\)
smooth such that \(F(x,\psi(x)) = 0\) and \(\phi(\ell) = L\). The function \(\chi\colon (0,\ell + \varepsilon_0)\to\mathbb{R}\) 
such that \(\chi(x) = \varphi(x)\) if \(x\in (0,\ell)\) and \(\chi(x) = \psi(x)\) if \(x\in [\ell,\ell + \varepsilon_0)\) 
is a smooth extension of \(\varphi\), contradicting the maximality of \((0,\ell)\).

The function \(\varphi\colon (-\varepsilon,+\infty)\) is \(C^2\), as \(F\in C^2 (\mathbb{R}^2)\), and strictly increasing
from \eqref{eq.ift}. We define \(V\colon [0,+\infty)\to\mathbb{R}\) as \(V(s) := \varphi(s^2)\). If \(s > 0\), then \(V'(s) = 2s\varphi'(s^2) > 0\).

(ii). From \eqref{eq.V}, \(\lim_{s\to\infty} V(s) = +\infty\). 
\end{proof}
Figure~\ref{fig.sigma-subcritical} and \ref{fig.sigma-supercritical} illustrate the behaviour of \(V\) when \(\sigma < 3\). 
\begin{center}
\begin{figure}[ht!]
\centering\includegraphics[scale=0.6]{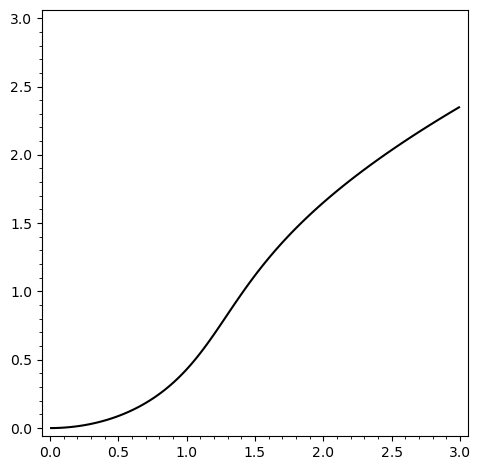}
\caption{\(a = 1\), \(b = -2\) and \(c = 3\). \(\sigma = 4/3 < 3\). \(V\) is defined everywhere.}
\label{fig.sigma-subcritical}
\end{figure}
\end{center}
\begin{prop}
\label{prop.lambda}
If \eqref{eq.C}, \eqref{eq.V} and \eqref{eq.G} hold, then \(\omega^{-1/2}\lambda'(\omega)\) is as a third degree polynomial function on \(\omega\).
\end{prop}
\begin{proof}
We derive an explicit expression of \(\lambda\) using the approached devised in \cite{IK93}, and adopted in \cite{GG17,GG19}.
Multiplying \eqref{eq.elliptic} by \(R_\omega\) and integrating, we obtain 
\begin{linenomath}
\begin{equation*}
R_\omega'(x)^2 - 2G(R_\omega(x)) - \omega R_\omega(x)^2\equiv C
\end{equation*}
\end{linenomath}
for some \(C\) not depending on \(x\). From the remarks following Proposition~\ref{prop.ber-lio-str}, \(C = 0\). 
From the equality above, (i) of Proposition~\ref{prop.ber-lio-str} and \eqref{eq.G} follows
\begin{linenomath}
\begin{equation}
\label{eq.lambda.2}
R_\omega'(x) = -\sqrt{2G(R_\omega(x))+\omega R_\omega^2(x)} =  
-R_\omega (x)\sqrt{\omega - V(R_\omega(x))}.
\end{equation}
\end{linenomath}
From Proposition~\ref{prop.branch}, \(V\colon [0,+\infty)\to [0,+\infty)\) is a surjective function. 
Let \(T\colon [0,+\infty)\to [0,+\infty)\) be the inverse of \(V\). Therefore,
\begin{linenomath}
\begin{equation*}
V(T(\omega)) = \omega.
\end{equation*}
\end{linenomath}
From the definition of \(\lambda\) in \eqref{eq.lambda.def}, (i) of Proposition~\ref{prop.ber-lio-str} and
\eqref{eq.lambda.2}, we have
\begin{linenomath}
\begin{equation*}
\begin{split}
\lambda(\omega) = \int_{-\infty}^{+\infty} R_\omega(x)^2 dx = 
2\int_{0}^{+\infty} R_\omega(x)^2 dx = 2\int_{0}^{+\infty} \frac{R_\omega(x)^2 R_\omega'(x)}{R_\omega'(x)} dx.
\end{split}
\end{equation*}
\end{linenomath}
From (i) of Proposition~\ref{prop.ber-lio-str}, \(R'\) vanishes only at \(x = 0\). Therefore, the integral above is well defined. From \eqref{eq.lambda.2} the last
term of the equality above is equal to
\begin{linenomath}
\begin{equation*}
-2\int_{0}^{+\infty} \frac{R_\omega(x)^2 R_\omega'(x)dx}{-R_\omega (x)\sqrt{\omega - V(R_\omega(x))}}
= -2\int_{0}^{+\infty} \frac{R_\omega(x) R_\omega'(x)dx}{\sqrt{\omega - V(R_\omega(x))}}.
\end{equation*}
\end{linenomath}
From \eqref{eq.lambda.2} \(R_\omega\colon [0,+\infty)\to (0,T(\omega)]\) is strictly decreasing and surjective.
The variable change \(s = R_\omega(x)\) yields
\begin{linenomath}
\begin{equation}
\label{eq.lambda.3}
\begin{split}
2&\int_0^{T(\omega)}\frac{sds}{\sqrt{\omega - V(s)}} = 
2\int_0^{T(\omega)} \frac{s}{V'(s)}\cdot\frac{V'(s)ds}{\sqrt{\omega - V(s)}}\\
2&\int_0^{T(\omega)} \frac{s}{V'(s)}\cdot\frac{d}{ds}\left(-2\sqrt{\omega - V(s)}\right)'ds.
\end{split}
\end{equation}
\end{linenomath}
The second integral is well-defined as \(V' > 0\) almost everywhere in \([0,+\infty)\), from Proposition~\ref{prop.branch}. 
A by-parts integration gives
\begin{linenomath}
\begin{equation*}
\begin{split}
-2&\int_0^{T(\omega)} \frac{d}{ds}\left(\frac{s}{V'(s)}\right)\cdot\left(-2\sqrt{\omega - V(s)}\right)ds\\
+ 2&\left[\frac{s}{V'(s)}\cdot -2
\sqrt{\omega - V(s)}\right]_0^{T(\omega)}\\
= 4&\int_0^{T(\omega)} \frac{d}{ds}\left(\frac{s}{V'(s)}\right)\cdot\sqrt{\omega - V(s)}ds + 2c\sqrt{\omega}.
\end{split}
\end{equation*}
\end{linenomath}
Taking the derivative with respect to \(s\) in \eqref{eq.implicit.1}, we obtain 
\begin{linenomath}
\begin{equation}
\label{eq.implicit.1}
2s = (3aV^2 + 2bV + c)V'.
\end{equation}
\end{linenomath}
On dividing by \(s\), and taking the limit, we obtain \(\lim_{s\to 0} s/V'(s) = c/2\). Therefore, the last equality follows
too.
Taking the derivative of \(\lambda\) with respect to \(\omega\), we obtain
\[
\begin{split}
\lambda'(\omega) &= 
2\int_0^{T(\omega)} \frac{d}{ds}\left(\frac{s}{V'(s)}\right)\cdot\frac{ds}{\sqrt{\omega - V(s)}} + \frac{c}{\sqrt{\omega}} \\
&= 
2\int_0^{T(\omega)} \frac{1}{V'(s)}\frac{d}{ds}\left(\frac{s}{V'(s)}\right)\cdot\frac{V'(s)ds}{\sqrt{\omega - V(s)}} + \frac{c}{\sqrt{\omega}} \\
&= -4\int_0^{T(\omega)} \frac{1}{V'(s)}\frac{d}{ds}\left(\frac{s}{V'(s)}\right)\cdot (\sqrt{\omega - V(s)})'ds + \frac{c}{\sqrt{\omega}}.
\end{split}
\]
We set 
\begin{linenomath}
\begin{equation}
\label{eq.implicit.4}
g(s) := \frac{1}{V'(s)}\frac{d}{ds}\left(\frac{s}{V'(s)}\right).
\end{equation}
\end{linenomath}
From \eqref{eq.implicit.1}, we obtain
\begin{linenomath}
\begin{equation}
\label{eq.implicit.2}
g(s) = 3aV + b.
\end{equation}
\end{linenomath}
Therefore, the limit of \(g\) as \(s\to 0\) exists, and it is equal
to \(b\). Then,
\[
\begin{split}
\lambda'(\omega) &= -4\int_0^{T(\omega)} g(s) (\sqrt{\omega - V(s)})'ds + \frac{c}{\sqrt{\omega}}\\
&= 4\int_0^{T(\omega)} g'(s)\sqrt{\omega - V(s)}ds + \frac{c}{\sqrt{\omega}} + \left[-4g(s)\sqrt{\omega - V(s)}\right]_0^{T(\omega)}\\
&= 4\int_0^{T(\omega)} g'(s)\sqrt{\omega - V(s)}ds + \frac{c}{\sqrt{\omega}} + 
\left[-4g(s)\sqrt{\omega - V(s)}\right]_0^{T(\omega)}\\
&= 4\int_0^{T(\omega)} g'(s)\sqrt{\omega - V(s)}ds + \frac{c}{\sqrt{\omega}} + 4b\sqrt{\omega}.
\end{split}
\]
Taking the derivative in \eqref{eq.implicit.2}, we obtain \(g'(s) = 3aV'\). Since
\begin{linenomath}
\begin{equation*}
\begin{split}
12a\int_0^{T(\omega)} V'\sqrt{\omega - V(s)}ds &= 
-8a\Big[(\omega - V)^{\frac{3}{2}}\Big]_0^{T(\omega)} = 
8a\omega^{\frac{3}{2}},
\end{split}
\end{equation*}
\end{linenomath}
there holds
\begin{linenomath}
\begin{equation}
\label{eq.lp.1}
\lambda'(\omega) = 8a\omega^{\frac{3}{2}} + 4b\omega^{\frac{1}{2}} + c\omega^{-\frac{1}{2}}.
\end{equation}
\end{linenomath}
Therefore, \(\omega^{-\frac{1}{2}}\lambda(\omega) = p(\omega)\), where
\begin{linenomath}
\begin{equation}
\label{eq.pol}
p = 8a X^2 + 4b X + c.
\end{equation}
\end{linenomath}
\end{proof}
\begin{rem}
Although the rightend side of \eqref{eq.lp.1} is defined on for every \(\omega\), it coincides with \(\|R_\omega\|_{L^2}^2\) 
only when \(\omega\) satisfies the assumptions of Proposition~\ref{prop.ber-lio-str}. In the case \(\sigma < 3\),
every \(\omega\) does satisfy those assumptions.
\end{rem}
\section{Stability of the ground-state and uniqueness}
\label{sect.stability.gs}
For the proof of Theorem~\ref{thm.stability.gs}, we rely on the results obtained in \cite{GG17}
which apply to more general non-linearities than those satisfying \eqref{eq.G}. We are going to
show that \(G\) satisfies assumptions (G1), (G2b) and (G4) necessary to apply \cite[Theorem~1.3]{GG17}.
Results in higher dimension have been proved under similar assumptions in \cite{BBGM07,Iko14}. 
For the sake of convenience, each of these assumptions are going to be restated before being checked.
\begin{proof}[Proof~of~Theorem~\ref{thm.stability.gs}]
\

\noindent (G1). There exists \(s_0 > 0\) such that \(G(s_0) < 0\).

\noindent From Proposition~\ref{prop.branch} and the assumption \(V(0)\) in \eqref{eq.V}, it follows that \(V(s) > 0\) for
every \(s > 0\). Therefore, from \eqref{eq.V}, \(G(s) < 0\) for every \(s > 0\).

\noindent (G2a). There are \(2 < p < q\) and \(C\) such that \(|G'(s)|\leq C(|s|^{p - 1} + |s|^{q - 1})\).

\noindent Since 
\begin{linenomath}
\begin{equation}
\label{eq.G2a.1}
3ax^2 + 2bx + c\geq (3 - \sigma)/3c
\end{equation}
\end{linenomath}
from \eqref{eq.implicit.1} we have
\begin{linenomath}
\begin{equation*}
\begin{split}
2s = |V'(s)| |3aV(s)^2 + 2V(s)b + c|\geq \frac{3 - \sigma}{3c} |V'(s)|.
\end{split}
\end{equation*}
\end{linenomath}
Therefore,
\begin{linenomath}
\begin{equation}
\label{eq.G2a.2}
|V'(s)|\leq\frac{6c}{3 - \sigma}|s|.
\end{equation}
\end{linenomath}
Integrating both the left and the righthand side of the inequality above, we obtain 
\begin{linenomath}
\begin{equation}
\label{eq.G2a.3}
|V(s)|\leq\frac{3c}{3 - \sigma} s^2.
\end{equation}
\end{linenomath}
From \eqref{eq.V},
\begin{linenomath}
\begin{equation*}
\label{eq.G2a.4}
G' = -\frac{1}{2}(2sV + s^2 V').
\end{equation*}
\end{linenomath}
Therefore
\begin{linenomath}
\begin{equation}
\label{eq.G2a}
\begin{split}
|G'(s)|&\leq |sV| + \frac{1}{2}|s|^2 |V'|\\
&\leq\frac{3c}{(3 - \sigma)}|s|^3 + \frac{3c}{(3 - \sigma)} |s|^3 = C|s|^3,\quad C := \frac{6c}{3 - \sigma}.
\end{split}
\end{equation}
\end{linenomath}
(G2b). There exists \(2 < p^* < 6\) and \(s^*\) such that for \(0 < s\leq s^*\) there holds
\(G(s)\geq -C|s|^{p^*}\)

\noindent From \eqref{eq.C}, \eqref{eq.V}
and \eqref{eq.G}, it follows that \(G(0) = 0\) which, together with \eqref{eq.G2a}, implies
\(|G(s)|\leq 3c/{2(3 - \sigma)} |s|^4\) after an integration on \([0,s]\). Therefore, (G2b) holds
with \(p^* = 4\) and every \(s^* > 0\).

The last assumption of \cite[Theorem~1.1]{GG17} is the global well-posedness of \eqref{eq.NLS} in \(H^1\).
From Proposition~\ref{prop.branch} and \eqref{eq.G}, \(G\) is \(C^2\). Therefore, \(G'\)
is locally Lipschitz. Therefore, from \cite[Theorem~3.5.1]{Caz03},
for every initial datum \(u_0\) in \(H^1 (\mathbb{R};\mathbb{C})\) there exists a unique \(T_{max}\) and  
\(\phi\colon [0,T_{max})\times\mathbb{R}\to\mathbb{C}\), solution to \eqref{eq.NLS}, such that 
\(\phi(t,\cdot)\) has regularity \(C([0,T_{max});H^1)\cap C^1([0,T_{max});L^2)\) and 
\(\phi(0,\cdot) = u_0\). From the conservation of the energy and the mass, and 
(G2b), \(\phi(t,\cdot)\) is bounded in \(H^1\). Therefore, local solutions can be extended to 
\([0,+\infty)\), showing that \eqref{eq.NLS} is globally well-posed. 
Since \(G\) satisfies assumptions (G1), (G2a) and (G2b), by \cite[Theorem~1.1]{GG17}, the set \(\mathcal{G}_\lambda\) is 
non-empty and stable for every \(0 < \lambda\). 
\end{proof}
\begin{prop}
If \(\sigma < 3\), a unique solution to \eqref{eq.elliptic} exists for every \(\omega\) in \((0,+\infty)\).
\end{prop}
\begin{proof}
From \eqref{eq.V}, \(V(0) = 0\). From Proposition~\ref{prop.branch}, \(V\) is strictly increasing
and diverging. Therefore, \(V\colon [0,+\infty)\to [0,+\infty)\) is surjective.
Therefore, for every \(\omega > 0\), 
there exists a positive, symmetric-decreasing solution to \eqref{eq.elliptic}, by Proposition~\ref{prop.ber-lio-str}.
\end{proof}
\begin{prop}
\label{prop.limit} 
\(\lim_{\omega\to 0} \lambda(\omega) = 0\).
\end{prop}
\begin{proof}
From \eqref{eq.implicit.2} and \eqref{eq.implicit.4} , \((s/V')' = V'(3aV + b)\). Therefore, \((s/V')'\)
is locally bounded on \([0,+\infty)\). From \eqref{eq.lambda.3} and previous calculations,
\begin{linenomath}
\begin{equation*}
\begin{split}
\lambda(\omega) &= 4\int_0^{T(\omega)} \frac{d}{ds}\left(\frac{s}{V'(s)}\right)\cdot\sqrt{\omega - V(s)}ds + 2c\sqrt{\omega}\\
&\leq 4\|(s/V')'\|_{L^{\infty}(0,T(\omega))} T(\omega)\sqrt{\omega} + 2c\sqrt{\omega}\sim\sqrt{\omega}
\end{split}
\end{equation*}
\end{linenomath}
as \(\omega\to 0^+\). Therefore, the statement of the proposition follows.
\end{proof}
\begin{proof}[Proof~of~Theorem~\ref{thm.uniqueness}]
\ 
(i). If \(0 < \sigma \leq 2\), the discriminant of the second degree polynomial in \eqref{eq.pol} is non-positive as
\(16b^2 - 32ac = 16ac(\sigma - 2)\leq 0\).
Therefore, \(\lambda'(\omega)\geq 0\) for every
\(\omega\). We claim that for every \(\lambda > 0\), there holds \(|K_\lambda| = 1\). On the contrary, 
let \(0 < \omega_1\leq\omega_2\) and \(R_1,R_2\in K_\lambda\) be such that 
\[
R_i '' - G'(R_i) - \omega_i R_i = 0,\quad R_i\in H^{1}_{r,+} (\mathbb{R};\mathbb{R}).
\]
Since
\(\lambda\) is monotonically non-decreasing on \([\omega_1,\omega_2]\), we have
\(\lambda(\omega_1)\leq\lambda(\omega_2)\). Since \(R_{1},R_{2}\in S(\lambda)\),
we have \(\lambda(\omega_1) = \lambda(\omega_2)\). If \(\omega_1 < \omega_2\), then \(\lambda'\equiv 0\)
on \([\omega_1,\omega_2]\). This implies \(a = b = c = 0\) in \eqref{eq.lp.1}, contradicting the assumption \(c > 0\) in \eqref{eq.V}.
Otherwise, we have \(\omega_1 = \omega_2\) and \(R_{1}\) and \(R_{2}\) are solutions in \(H^1 _{r,+}(\mathbb{R})\)
to \eqref{eq.elliptic} with \(\omega := \omega_1\). From Proposition~\ref{prop.ber-lio-str}, this problem has a unique solution,
therefore \(R_{1} = R_{2}\) proving that \(|K_\lambda| = 1\). 
\begin{figure}[!ht]
\centering\includegraphics[scale=0.6]{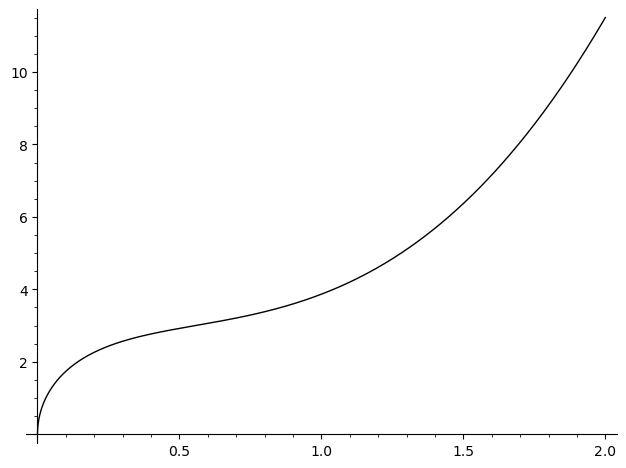}
\caption{\(a = 1\), \(b = -2\) and \(c = 3\). \(\sigma = \frac{4}{3} < 2\).}
\end{figure}
(ii). When \(2 < \sigma < 3\) there are 
\begin{linenomath}
\begin{equation}
\label{eq.omoM}
\omega_m < \omega_M
\end{equation}
\end{linenomath}
such that \(\lambda'(\omega_m) = \lambda'(\omega_M) = 0\). We set \(\lambda_M := \lambda(\omega_m)\)
and \(\lambda_m := \lambda(\omega_M)\). Clearly, \(\lambda_m < \lambda_M\). From Proposition~\ref{prop.limit} and \eqref{eq.lp.1}, 
\begin{linenomath}
\begin{equation}
\label{eq.lp.3}
\lambda(\omega) = \frac{16a}{5}\omega^{\frac{5}{2}} + \frac{8b}{3}\omega^{\frac{3}{2}} + 2c\omega^{\frac{1}{2}}.
\end{equation}
\end{linenomath}
By the Intermediate Value Theorem, for every \(\lambda\in (\lambda_m,\lambda_M)\), there are frequencies
\(\omega_1(\lambda) < \omega_2 (\lambda) < \omega_3 (\lambda)\) such that \(\lambda(\omega_i (\lambda)) = \lambda\) and
\begin{linenomath}
\begin{equation*}
\omega_1 (\lambda)\in (0,\omega_m),\quad \omega_2 (\lambda)\in (\omega_m,\omega_M),\quad \omega_3(\lambda)\in (\omega_M,+\infty).
\end{equation*}
\end{linenomath}
From the Implicit Function Theorem, \(\omega_i\in C^1\) in their domain and
\begin{linenomath}
\begin{equation}
\label{eq.omega-derivative}
\omega_1' (\lambda),\omega_3'(\lambda) > 0,\quad \omega_2' (\lambda) < 0.
\end{equation}
\end{linenomath}
We define
\begin{linenomath}
\begin{equation*}
g_1 (\lambda) := \int_{\omega_1 (\lambda)}^{\omega_2 (\lambda)} 
(\lambda(\omega) - \lambda)d\omega,\quad
g_2 (\lambda) := \int_{\omega_2 (\lambda)}^{\omega_3 (\lambda)} (\lambda - \lambda(\omega))d\omega.
\end{equation*}
\end{linenomath}
The function \(g_1\colon (\lambda_m,\lambda_M)\to (0,+\infty)\) is strictly decreasing,
\(g_1 (\lambda_M) = 0\), while
\(g_2\colon (\lambda_m,\lambda_M)\to(0,+\infty)\) is strictly increasing and \(g_2 (\lambda_m) = 0\). Therefore, there exists a unique \(\lambda_2\) such that 
\((g_2 - g_1)(\lambda_2) = 0\), and
\begin{linenomath}
\begin{equation}
\label{eq.uniqueness.4}
g_2 - g_1 > 0 \text{ on } (\lambda_2,\lambda_M),\quad g_2 - g_1 < 0 \text{ on } (\lambda_m,\lambda_2).
\end{equation}
\end{linenomath}
We define the
function \(e\colon (0,+\infty)\to\mathbb{R}\) as \(e(\omega) := E(R_\omega)\). In the proof of Theorem~\ref{thm.stability.gs}, it has been shown that \(G\) satisfies
the assumption \cite[(G2a)]{GG17}. Therefore, from \cite[Proposition~7]{GG17}, \(E\) is \(C^1(H^1,\mathbb{R})\). From (iii) of Proposition~\ref{prop.ber-lio-str} \(e\)
is also continuously differentiable; \eqref{eq.elliptic} reads 
\begin{linenomath}
\begin{equation*}
E'(R_\omega) = -\frac{\omega}{2} M'(R_\omega).
\end{equation*}
\end{linenomath}
Therefore,
\begin{linenomath}
\begin{equation*}
\label{eq.uniqueness.0}
e'(\omega) = \langle E'(R_\omega),\phi'(\omega) \rangle = \langle -\frac{\omega}{2} M'(R_\omega),\phi'(\omega)\rangle = 
-\frac{\omega}{2}\lambda'(\omega).
\end{equation*}
\end{linenomath}
For every \(\lambda\in (\lambda_m,\lambda_M)\), 
\(E\) has three critical points constrained to \(H^1 _{r,+}\), namely \(R_{\omega_1},R_{\omega_2}\) and \(R_{\omega_3}\). In order
to establish which ones are in \(\mathcal{G}_\lambda\), we compare pairwise their energy.
By integrating \eqref{eq.uniqueness.0} on the open intervals \((\omega_1,\omega_3)\), \((\omega_2,\omega_3)\) and 
\((\omega_1,\omega_3)\), we obtain
\begin{linenomath}
\begin{equation}
\label{eq.uniqueness.1}
\begin{split}
E(R_{\omega_3}) - E(R_{\omega_1}) &= e(\omega_3) - e(\omega_1) \\
&= \int_{\omega_1(\lambda)}^{\omega_3(\lambda)} e'(\omega)d\omega = -\frac{1}{2}\int_{\omega_1(\lambda)}^{\omega_3(\lambda)} \omega \lambda'(\omega)d\omega \\
&= \frac{1}{2}\int_{\omega_1(\lambda)}^{\omega_3(\lambda)} \lambda(\omega)d\omega - \frac{1}{2}\lambda(\omega_3 - \omega_1) \\
&= \frac{1}{2}(g_1 - g_2)(\lambda),
\end{split}
\end{equation}
\end{linenomath}
\begin{linenomath}
\begin{equation}
\label{eq.uniqueness.2}
\begin{split}
E(R_{\omega_2}) - E(R_{\omega_1}) &= e(\omega_2) - e(\omega_1) \\ 
&= \int_{\omega_1(\lambda)}^{\omega_2(\lambda)} e'(\omega)d\omega = 
-\frac{1}{2}\int_{\omega_1(\lambda)}^{\omega_2(\lambda)} \omega\lambda'(\omega)d\omega \\
&= \frac{1}{2} 
\int_{\omega_1(\lambda)}^{\omega_2(\lambda)} \lambda(\omega)d\omega - \frac{1}{2}\lambda(\omega_2 - \omega_1) \\
&= \frac{1}{2} g_1(\lambda) \geq 0.
\end{split}
\end{equation}
\end{linenomath}
Similarly, the Fundamental Theorem of Calculus and a by-parts integration yields
\begin{linenomath}
\begin{equation}
\label{eq.uniqueness.3}
\begin{split}
E(R_{\omega_3}) - E(R_{\omega_2}) &= e(\omega_3) - e(\omega_2) \\ 
&= \frac{1}{2}\int_{\omega_2(\lambda)}^{\omega_3(\lambda)} \lambda(\omega)d\omega - \frac{1}{2}\lambda(\omega_3 - \omega_2) \\
&= -\frac{1}{2}g_2(\lambda) \leq 0.
\end{split}
\end{equation}
\end{linenomath}
respectively. From \eqref{eq.uniqueness.4}, \eqref{eq.uniqueness.1}, \eqref{eq.uniqueness.2} and \eqref{eq.uniqueness.3} we are able to prove the main statement
of the theorem. 

\subsubsection*{First case: \(\lambda_m \leq \lambda < \lambda_2\)} 

From \eqref{eq.uniqueness.4}, \eqref{eq.uniqueness.1} and \eqref{eq.uniqueness.3},
\begin{linenomath}
\begin{equation}
\label{eq.uniqueness.5}
E(R_{\omega_1}) <  E(R_{\omega_3}) \leq E(R_{\omega_2}).
\end{equation}
\end{linenomath}
Therefore, 
\(|K_\lambda| = 1\) as it contains only \(R_{\omega_1}\).
\subsubsection*{Second case: \(\lambda_2 < \lambda \leq \lambda_M\)}
\begin{linenomath}
\begin{equation}
\label{eq.uniqueness.6}
E(R_{\omega_3}) <  E(R_{\omega_1}) \leq E(R_{\omega_2})
\end{equation}
\end{linenomath}
Therefore, \(|K_\lambda| = 1\) as it contains only \(R_{\omega_1}\).
\subsubsection*{Third case: \(\lambda = \lambda_2\)} \(E(R_{\omega_1}) = E(R_{\omega_3}) < E(R_{\omega_2})\). Therefore, \(|K_\lambda| = 2\): there are two positive, 
symmetric minima of \(E\) constrained to \(S(\lambda)\).
\begin{figure}[!ht]
\centering\includegraphics[scale=0.6]{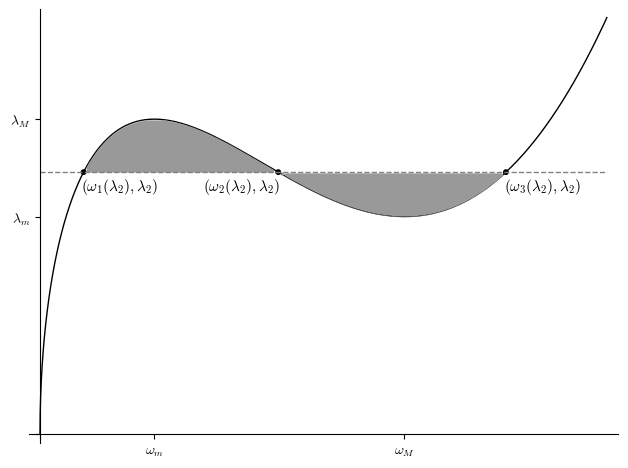}
\caption{\(a = 1\), \(b = -\sqrt{33}/2\) and \(c = 3\). \(\sigma = 2 < 11/4 < 3\). The two grey color filled
regions have measure \(g_1(\lambda_2)\) and \(g_2 (\lambda_2)\). Since they coincide, in \(S_r (\lambda_2)\) there
are two positive minima.}
\end{figure}
\end{proof}
\begin{proof}[Proof~of~Theorem~\ref{thm.stability}]
Since \(|K_\lambda|\leq 2\) for every \(\lambda > 0\), from \cite[Proposition~5]{GG17}, there
are finitely many \(\mathcal{G}_\lambda (u)\) as \(u\) varies in \(\mathcal{G}_\lambda\).
Therefore, from \cite[Theorem~1.3]{GG17}, \(\mathcal{G}_\lambda (u)\) is orbitally stable for every \(\lambda > 0\)
and for every \(u\in\mathcal{G}_\lambda\).
\end{proof}
\section{Non-degeneracy of minima}
Firstly, we are going to show that \(E\) is \(C^2(H^1(\R;\C),\R)\). Therefore, the problem of degeneracy or non-degeneracy of the Hessian at a given critical point is well defined. While the gradient term of \(E\) is analytic, the regularity
of the integral part of \(E\) relies on the following assumption introduced
in \cite[(G4)]{GG17} and \cite{BBBM10}: \(G\) is two-times continuously differentiable, \(G(0) = G'(0) = 0\) and there exist \(C\) and \(2 < p < q\) such that
\begin{linenomath}
\begin{equation*}
|G''(s)|\leq C(|s|^{p - 2} + |s|^{q - 2}).
\end{equation*}
\end{linenomath}
A proof of the regularity of \(E\) is in \cite[Proposition~7]{GG17}. However, under the assumptions \eqref{eq.C},
\eqref{eq.G} and \eqref{eq.V}, \(G\) does not satisfy it. In fact, (G4) implies \(G''(0) = 0\), which does not hold
for every choice of \(a,b\) and \(c\) in \eqref{eq.V}. We workaround this obstacle in the proof of the next proposition.
\begin{prop}
If \eqref{eq.C}, \eqref{eq.G} and \eqref{eq.V} hold, then \(E\) is \(C^2(H^1(\R;\C),\R)\), provided \(\sigma < 3\).
\end{prop}
\begin{proof}
We can write
\begin{linenomath}
\begin{equation}
\label{eq.nondeg.4}
G(s) = G(s) - \frac{1}{2} G''(0)s^2 + \frac{1}{2} G''(0)s^2 =: G_0 + G_1. 
\end{equation}
\end{linenomath}
Taking the derivative with respect to \(s\) in \eqref{eq.implicit.1}, we obtain
\begin{linenomath}
\begin{equation*}
2 = (6aV + b){V'}^2 + (3aV^2 + 2bV + c)V''.
\end{equation*}
\end{linenomath}
Since \(\sigma < 3\), there exists \(s_0\geq 1\) such that \(6aV + b\geq 0\) for every \(s\geq s_0\), by
(ii) of Proposition~\ref{prop.branch}; \(2\geq |(3aV^2 + 2bV + c)V''|\geq\frac{3 - \sigma}{3c} |V''|\), 
by \eqref{eq.G2a.1}. Therefore
\begin{linenomath}
\begin{equation}
\label{eq.nondeg.2}
|V''(s)|\leq\frac{6c}{3 - \sigma},\quad s\geq s_0.
\end{equation}
\end{linenomath}
Taking the second derivative in \eqref{eq.G}, we obtain
\begin{linenomath}
\begin{equation*}
G'' = -\frac{1}{2}\left(2V + 4sV' + V''\right).
\end{equation*}
\end{linenomath}
From \eqref{eq.G2a.3}, \eqref{eq.G2a.2} and \eqref{eq.nondeg.2}, we obtain
\begin{linenomath}
\begin{equation*}
|G''(s)|\leq\frac{18c}{3 - \sigma} s^2,\quad s\geq s_0.
\end{equation*}
\end{linenomath}
Therefore, for every \(s\geq s_0\), 
\begin{linenomath}
\begin{equation}
\label{eq.nondeg.3}
\begin{split}
|G_0 ''(s)|&\leq |G''(0)| + |G''(s)|\\
&\leq  \frac{1}{2} |G''(0)|s^2 + \frac{18c}{3 - \sigma} s^2\leq
C s^2,
\end{split}
\end{equation}
\end{linenomath}
for some \(C\geq 0\).
Since \(G_1''(0) = 0\), from \eqref{eq.nondeg.3}, \(G_1\) fulfils (G4) with \(p = 3\) and \(q = 4\).
From \eqref{eq.E}, \eqref{eq.M} and \eqref{eq.nondeg.4}, we can write
\begin{linenomath}
\begin{equation*}
E(u) = \frac{1}{2}\int_{-\infty}^{+\infty}|\nabla u|^2 dx + \frac{1}{2} M(u) + \int_{-\infty}^{+\infty} G_0 (u)dx.
\end{equation*}
\end{linenomath}
The first two terms are \(C^\infty (H^1,\R)\) and the third term is \(C^2 (H^1,\R)\)
by \cite[Proposition~7]{GG17}.
\end{proof}
\begin{proof}[Proof~of~Theorem~\ref{thm.nondeg}]
To prove this theorem, we rely on the consequences of the calculations done in \cite[Proof~of~Theorem~1.2,~\S3]{GG17}: 
given \(\omega_0\) such that \(R_{\omega_0}\) is a minimum, \(R_{\omega_0}\) is non-degenerate if and only if  \(\lambda'(\omega_0) = 0\).
From \eqref{eq.lp.1}, 
\begin{linenomath}
\begin{equation*}
\lambda'(\omega) = 
\omega^{-\frac{1}{2}}(8a\omega^2 + 4b\omega + c).
\end{equation*}
\end{linenomath}
This function has a positive double root if and only if \(16b^2 - 32ac = 0\),
that is \(\sigma = 2\). Therefore, if \(0 < \sigma < 2\), then \(\lambda'(\omega) > 0\) and all the minima are non-degenerate.
If \(2 < \sigma < 3\), then \(\lambda'\) has two zeroes, \(\omega_m\) and 
\(\omega_M\), introduced in \eqref{eq.omoM}. On the constraint \(S(\lambda_M)\), there are two
critical points, \(R_{\omega_m} = R_{\omega_1} = R_{\omega_2}\) and \(R_{\omega_3}\). From the \eqref{eq.uniqueness.5},
\(E(R_{\omega_m}) > E(R_{\omega_3})\), so \(E(R_{\omega_m})\) is not a minimum; similarly, 
on the constraint \(S(\lambda_m)\), there are two critical points, \(R_{\omega_1}\) and \(R_{\omega_M} = R_{\omega_2} = R_{\omega_3}\). From \eqref{eq.uniqueness.6},
\(E(R_{\omega_M}) > E(R_{\omega_1})\), so \(E(R_{\omega_M})\) is not a minimum. Therefore,
there are not degenerate minima in the case \(2 < \sigma < 3\) as well.

Finally, for a non-linearity \(G\) satisfying \eqref{eq.G} and \eqref{eq.V}, if \(\sigma = 2\), then \(\lambda'\) has a single positive double root 
given by 
\(\omega_d := -b/4a\). 
From Theorem~\ref{thm.uniqueness}, \(|K_\lambda| = 1\). Therefore,
\(R_{\omega_d}\) is a minimum, degenerate. We set \(\lambda_d := \|R_{\omega_d}\|_2^2\).
Therefore, \(E\) has a positive, degenerate minimum on \(S(\lambda_d)\).
\begin{figure}[!ht]
\centering\includegraphics[scale=0.6]{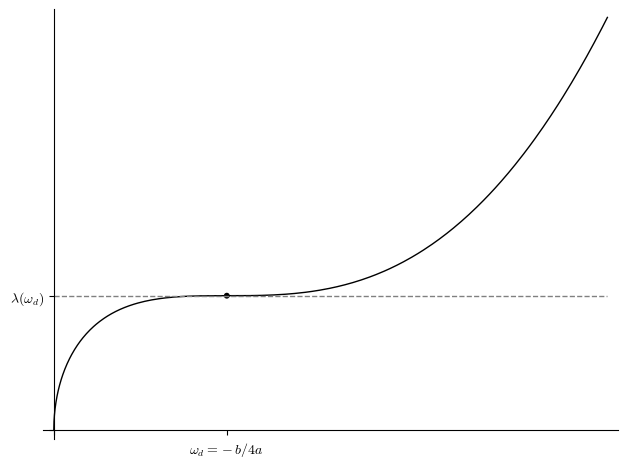}
\caption{\(a = 2\), \(b = -2\) and \(c = 1\). \(\sigma = 2\). \(\omega_d = 1/4\)
and \(R_{1/4}\) is a degenerate minimum of \(E\) constrained to \(S_r (\lambda(1/4))\).}
\end{figure}
\end{proof}
\begin{rem}
If \(\sigma\geq 3\), the existence domain of the function \(V\) cannot be extended beyond the bounded interval 
\((0,\ell^2)\), where \(\ell\) has been defined in \eqref{eq.branch.1}. In fact, \(\lim_{s\to \ell^2} V'(s) = +\infty\).
In principle, this is not an obstacle to defining the energy functional as in \eqref{eq.E}, after extending \(V\) as a \textsl{continuous} function on \([0,+\infty)\) and 
defining \(G\) as in \eqref{eq.G}. In this work, however, we preferred to focus on non-linearities \(G\) which are at least continuously differentiable, and thus falling into the variational setting studied in \cite{GG17} and other references as 
\cite{BBGM07,Iko14}.
\end{rem}
\begin{figure}[!ht]
\centering\includegraphics[scale=0.6]{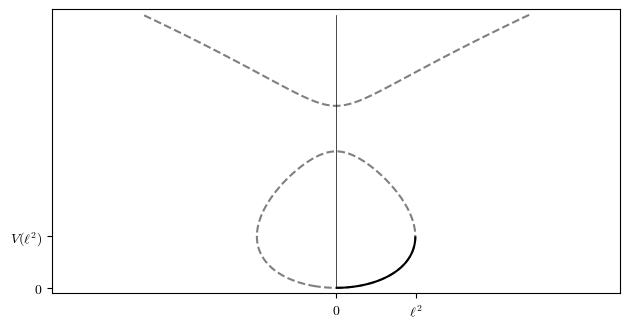}
\caption{\(a = 1\), \(b = -7/2\) and \(c = 3\). \(\sigma = 49/12 > 3\). The graph of \(V\) is 
represented by the solid black line. \(V\) cannot be extended smoothly beyond \((0,\ell^2)\), \(\ell\) being
defined in \eqref{eq.branch.1}.}
\label{fig.sigma-supercritical}
\end{figure}
\newpage
\def\cprime{$'$} \def\cprime{$'$} \def\cprime{$'$} \def\cprime{$'$}
  \def\cprime{$'$} \def\cprime{$'$} \def\cprime{$'$} \def\cprime{$'$}
  \def\cprime{$'$} \def\polhk#1{\setbox0=\hbox{#1}{\ooalign{\hidewidth
  \lower1.5ex\hbox{`}\hidewidth\crcr\unhbox0}}} \def\cprime{$'$}
  \def\cprime{$'$} \def\cprime{$'$}
\providecommand{\bysame}{\leavevmode\hbox to3em{\hrulefill}\thinspace}
\providecommand{\MR}{\relax\ifhmode\unskip\space\fi MR }
\providecommand{\MRhref}[2]{%
  \href{http://www.ams.org/mathscinet-getitem?mr=#1}{#2}
}
\providecommand{\href}[2]{#2}

\end{document}